\documentclass[11pt]{article}
\usepackage{amsmath}
\usepackage{amsfonts}
\usepackage{color}
\usepackage{latexsym}
\usepackage{tablefootnote}
\usepackage{epsfig}
\usepackage{multirow}
\usepackage{float}
\usepackage{array}
\usepackage{amssymb,amsthm}
\usepackage{hyperref}    % Hyperlink inside the paper 
\usepackage{algorithm2e} % For algorithms
\usepackage{cleveref}    % Automatic citation of environnements

\newcommand*{\Q}{\mathcal{Q}}
\newcommand*{\V}{\mathcal{V}}

\newcommand*{\E}{\mathbb{E}}

\definecolor{antiquefuchsia}{rgb}{0.57, 0.36, 0.51}
\definecolor{MyViolet}{rgb}{0.45,0.08,0.95}
\definecolor{MyBrown}{rgb}{0.45,0.08,0}
\definecolor{MyDarkBlue}{rgb}{0,0.08,0.45}
\usepackage{makecell,booktabs}
\usepackage[version=4]{mhchem}   
\usepackage[strict]{changepage}

\usepackage{amsmath}
\usepackage{amsfonts}
\usepackage{latexsym}
\usepackage{amssymb}

\newtheorem{thm}{Theorem}[section]
\newtheorem{theorem}[thm]{Theorem}

\newtheorem{lemma}[thm]{Lemma}

\newtheorem{definition}[thm]{Definition}

\newtheorem{rem}[thm]{Remark}
\newtheorem{remark}[thm]{Remark}

\newcommand{\beq}{\begin{equation}}
\newcommand{\eeq}{\end{equation}}
\newcommand{\beqa}{\begin{eqnarray}}
\newcommand{\eeqa}{\end{eqnarray}}
\newcommand{\beqas}{\begin{eqnarray*}}
\newcommand{\eeqas}{\end{eqnarray*}}
\newcommand{\bi}{\begin{itemize}}
\newcommand{\ei}{\end{itemize}}

\newcommand{\R}{\mathbb{R}}

\def\iff{\Leftrightarrow}

\begin{document}

	\title{Bidirectional SDDP with dimension-free complexity 
    for solving strongly convex stochastic dynamic programming equations}

	 \maketitle

     \begin{tabular}{cc}
\begin{tabular}{c}
Pablo Barros\\
School of Applied Mathematics, FGV\\
Praia de Botafogo, Rio de Janeiro, Brazil\\
{\tt pabloacbarros@gmail.com}
\end{tabular}
&
\begin{tabular}{c}
Vincent Guigues\\
School of Applied Mathematics, FGV\\
Praia de Botafogo, Rio de Janeiro, Brazil\\
{\tt vincent.guigues@fgv.br}
\end{tabular}\\
\end{tabular}
  
     \begin{abstract}
We analyze the complexity of Bidirectional Stochastic Dual Dynamic Programming (BSDDP) algorithm applied to multistage stochastic optimization problems with strongly convex cost functions. Under standard regularity assumptions and without discount factor, we establish an explicit complexity bound on the expected number of iterations needed to obtain an $\bar{\varepsilon}$-optimal policy. Our bound depends on the strong convexity constants and Lipschitz constants of the stagewise cost functions, but is independent of the dimensionality of the state space. This improves upon the complexity of SDDP for convex problems from \cite{lan2020}, particularly when the number of stages is moderate and the state dimension is large. The analysis further reveals that the complexity decreases as the strong convexity of the cost functions increases, highlighting the algorithm’s efficiency in structured problem instances.
\end{abstract}
 
		{\bf Keywords:} strongly convex  optimization, iteration-complexity, Dynamic Programming, SDDP.\\
		
		% 	{\bf Mathematics Subject Classification (2010)} 
		% 	49M37 $\cdot$ 65K05 $\cdot$ 68Q25 $\cdot$ 90C25 $\cdot$ 90C30 $\cdot$ 90C60
		{\bf AMS subject classifications:} 90C15, 90C90.\\ 
        
\section{Introduction}

In this paper, we introduce the \textit{Bidirectional Stochastic Dual Dynamic Programming (BSDDP)} algorithm applied to Multistage Stochastic Optimization Problems (MSOPs) with strongly convex cost functions and we analyze the complexity of this algorithm. Our setting involves $T$ stages and no discount factor.
BSDDP is a variant of SDDP that computes trial points in forward passes, same as SDDP, but that delays the computation of the cuts
to the next iteration where the same scenario is visited.
Therefore, we look backward in the history of scenarios to compute the cuts.
However, we also use a forward information on the set of sampled scenarios since both $k$th and $(k+1)$th sampled scenarios
are used at iteration $k$
whereas only $k$th scenario is used
in iteration $k$ of SDDP.
The name Bidirectional of our BSDDP method comes from the use of this forward-backward information on scenarios (to distinguish from the forward and backward passes which are still present as in SDDP). The motivation of BSDDP comes from our analysis of the complexity of SDDP and of our extension of
the complexity analysis of \cite{fddpguibarrren26} to the stochastic setting. 

SDDP method was coined in \cite{pereira} as a sampling-based extension
of the Nested Decomposition method from \cite{birgemulti}. Since then, many enhancements
of this method, which is still  the most popular to solve MSOPs, have been proposed, for instance extensions to  risk-averse problems
\cite{philpmatos}, \cite{guiguesrom10}, convergence proof
for risk-neutral \cite{lecphilgirar12}  and risk-averse \cite{guiguessiopt2016} convex programs, inexact variants  \cite{guigues2016isddp}
\cite{gms2020isddp}, variants with a random number of stages 
\cite{gui21r}, or variants for nonconvex problems  such as
\cite{midas}, \cite{ceri12}, see also \cite{shadenrbook, shapsddp, sddprebeneck} 
for a review.

Many papers develop complexity bounds for static or two-stage stochastic optimization problems such as \cite{nemjudlannem09}, \cite{singlestoone} but we are only aware of \cite{lan2020}
for the study of the complexity of SDDP applied to
MSOPs. In the derministic case,
a dimension-free complexity bound was derived
in \cite{fddpguibarrren26} for a framework
that contains Dual Dynamic Programming as a special
case. In this paper, we extend the tools
developed in \cite{fddpguibarrren26} 
to the stochastic setting. This includes
the introduction of several new key quantities ($\tilde V_t^k$
and 
$\tilde Q_t^k$
in \eqref{Vtildetk},
$e_{t}^k$ and 
$\tilde e_t^k$ in
\eqref{etkdef},
$\Delta_t^{k,\ell}$ in
\eqref{defdeltatk}, and event
$E^k$ in \eqref{defEventEk}), the
use of the pigeonhole principle, the derivation of
several new key inequalities, and of the use of Markov chains.

We introduce BSDDP and establish an iteration complexity bound for BSDDP which improves 
the complexity from \cite{lan2020} of SDDP applied to convex, not necessarily strongly convex, problems. Notably, our bound does not depend on the dimension of the state space, unlike the convex (non-strongly convex) case analyzed in \cite{lan2020}.

More precisely, in \cite{lan2020}, the expectation $\mathbb{E}[K]$  of the number $K$  of iterations required by SDDP to produce an $\bar{\varepsilon}$-optimal solution  for a problem with $T$ stages and without discount factor
satisfies
\begin{equation}\label{boundlan}
\mathbb{E}[K] \le \bar{K} \cdot \bar{N} + 2,
\end{equation}
where \(\bar{K} := \sum_{t=1}^{T-1} \left( \frac{D_t}{\delta_t} + 1 \right)^{n_t}\) and \(\bar{N} := \max\{N_1, \ldots, N_T\}.\)  
In this expression, $n_t$ is the size of the state vector for stage $t$, $N_t$ is the number of possible realizations for the random vector $\xi_t$ for stage $t$, and 
$D_t$ is the diameter of some set $X_t$ containing
the decisions for stage $t$,

In contrast, we show that under strong convexity, the expectation $\mathbb{E}[K]$  of the number $K$  of iterations required by BSDDP to produce an $\bar{\varepsilon}$-optimal solution  for a problem with $T$ stages and without discount factor
satisfies
\begin{equation}\label{ourbound}
    \mathbb{E}[K] \le \sum_{i=1}^R \frac{1}{p^i}, \quad R = 
1 + \left( \prod_{t=2}^T N_t \right) \cdot 
\left\lceil\left(
        1+ \left[ \frac{40}{\bar \varepsilon} \sum_{t=1}^{T-1} \frac{M_t^2}{\mu_t} \right]^{2^{T-1}-1}
    \right) \log\left( \frac{4}{\bar \varepsilon} \sum_{t=1}^{T-1} M_t D_t \right) \right \rceil
\end{equation}
where $M_t$ and $\mu_t$ denote the Lipschitz and strong convexity constants of the stage-$t$ cost functions and $p := \prod_{t=1}^{T-1} \min_{\tilde \xi_t \in \Xi_t} \mathbb{P}(\boldsymbol \xi_t= \tilde \xi_t)$
with $\Xi_t$ the support of $\xi_t$.

Our result highlights two key improvements: 
\begin{itemize}
    \item[(i)] the complexity bound is independent of the state dimensions $n_t$, and 
    \item[(ii)] it decreases as the strong convexity constants $\mu_t$ increase.
\end{itemize}
Hence, for problems with a large value of $n$ or when strong convexity is present, our bound \eqref{ourbound} is sharper than \eqref{boundlan}.

The outline of the paper is as follows: in Section \ref{secpbform}, we state the problem we want to solve and its assumptions.
In Section \ref{secsddp}, we state the BSDDP algorithm while
our complexity analysis of BSDDP applied to strongly convex problems  is given in Section \ref{secstronglyconvex}.

Throughout the paper, $\R^n$ denotes the standard $n$-dimensional Euclidean space equipped with  inner product 
$\langle x,y\rangle=x^T y$
with corresponding induced  norm denoted by 	$\|\cdot\|:=\|\cdot\|_2$.

\section{Problem formulation}\label{secpbform}

We consider a multistage stochastic optimization problem over $T$ stages. At each stage $t$, decisions are made based on the current realization of uncertainty and past decisions, with the goal of minimizing the expected total cost. We assume that the first-stage cost function is strongly convex and Lipschitz continuous, and that subsequent stagewise cost functions satisfy similar structural properties under each realization of uncertainty.

\subsection{A functional class}

To formalize the structure of the stagewise cost functions, we introduce a class of functions that encapsulates the regularity conditions required throughout the analysis.

\begin{definition}\label{defclassF}
    For $M, \mu>0$,
    let ${\cal F}(M,\mu)$ denote the class of
    functions $G:\mathbb{R}^m \times \mathbb{R}^n \to \mathbb{R}$ satisfying:
\begin{itemize}
    \item[i)] $G(\cdot,\cdot)$ is convex;
    \item[ii)] for every $y$,  $G(y,\cdot)$ is $M$-Lipschitz continuous, i.e.,
\[
G(y,x) - G(y,x') \le M\|x-x'\|, \quad \forall x, x'\in \mathbb{R}^{n};
\]
\item[iii)]
for every $x \in \mathbb{R}^n$, function $G(\cdot,x)$ is $\mu$-strongly convex;
    \item[iv)] for every $(y,x) \in \mathbb{R}^m \times \mathbb{R}^n$, a subgradient oracle $G'_x:\mathbb{R}^m \times \mathbb{R}^n \to \mathbb{R}^{n}$ satisfying
    \[
    G_x'(y,x) \in \partial_x G(y,x), \quad \|G'_x(y,x)\| \le M
    \]
    is available.
\end{itemize}
\end{definition}

This class captures the core assumptions needed to ensure well-posedness of the dynamic programming recursion and to derive complexity bounds in the presence of strong convexity.

We will denote by $\overline{\text{Conv}}(Y)$ the set of convex, proper, lower semicontinuous functions defined on
convex set $Y$.

\if{
\begin{lemma}\label{subgradValueFunc}
    Assume that
$G \in {\cal F}(M,\mu)$ and $ H \in \overline{\emph{Conv}}(Y)$. Consider the following optimization problem parameterized by $x \in X$:
\begin{equation}\label{pbdefsc}
            V(x) = \min_{y\in Y}\left\{G(y,x) + H(y) \right\}.
\end{equation}
        Then, the following statements hold:
\begin{itemize}
            \item[a)] problem \eqref{pbdefsc} has a unique optimal solution $y(x) \in Y$ and
        its optimal value
 $V(x)$
        is an $M$-Lipschitz continuous convex function in $X$;
        \item[b)]
        the
function $V': X \to \mathbb{R}^{d_x}$ defined as
\[
V'(x) := G'_x(y(x),x) 
\]
satisfies
\[
V'(x) \in  \partial V(x), \quad \|V'(x) \| \le M \quad \forall x \in X,
\]
where $y(x)$ is the optimal solution for \eqref{pbdefsc}.

\end{itemize}
\end{lemma}

\begin{proof} a) The first claim follows form the assumption that $G \in {\cal F}(M,\mu)$ and Definition \ref{defclassF} (iii).
It follows from Proposition 3.3.3 of \cite{bertsekas2009convex} that $V$ is a convex function.
We now show that $V$ is
$M$-Lipschitz continuous.
Let $y$ and $y'$ denote $y(x)$ and $y(x')$, respectively. Then, we have
\[
V(x) = G(y,x) + H(y), \quad V(x') = G(y',x') + H(y').
\]
Using the above identities, the minimality of $y$ and $y'$, and the fact that $G \in {\cal F}(M,\mu)$, we have
    \[
        V(x)-V(x') \le G(y',x) + H(y') - [G(y',x') + H(y')] = G(y',x) - G(y',x') \le M\|x-x'\|
        % V(x')-V(x) &\ge G(y',x') + H(y') - [G(y',x) + H(y')] = G(y',x') - G(y',x) \ge -M\|x-x'\|,
    \]
    and hence that $V$ is $M$-Lipschitz continuous.
    
    \par b) This statement immediately follows from Theorem 24(a) of \cite{rockafellar1974conjugate} and Definition \ref{defclassF} (iv).
\end{proof}

}\fi

Throughout this paper,
$T$ denotes the number of stages for the multistage stochastic problem stated below.
In addition to a compact convex set $X_1$ and a strongly convex function $F_1:X_1 \to \mathbb{R}$ that is
$M_1$-Lipschitz continuous on $X_1$, it is also
assumed that,
for each $t=2,\ldots,T$,
we have the following:
\begin{itemize}
     \item[(A1)]
     $X_t$ is a nonempty compact convex set;
    \item[(A2)] 
    $\xi_t$ is a random variable with finite support $\Xi_t$, of size $N_t$, and distribution given by  $\mathbb{P}_t.$
    \item[(A3)]
    for some $M_{t-1},\mu_t>0$, $F_t(\cdot,\cdot \, ;\xi_t) \in {\cal F}(M_{t-1},\mu_t)$ for every $\xi_t\in \Xi_t$.
\end{itemize}

Consider the following problem
\begin{equation}
\label{pb:multistage0-sto}
\mathcal{Q}_0^*=\left\{
\begin{array}{l}
    \min \quad  \mathbb{E} \left[F_1(x_1)+F_2(x_2(\xi_{[2]}),x_1;\xi_2) + \ldots + F_T(x_T(\xi_{[T]}),x_{T-1}(\xi_{[T-1]});\xi_T) \right],\\
    \text{s.t.} \quad x_1 \in X_1,\,
    x_2(\xi_{[2]}) \in X_2,\ \ldots ,\
    x_T(\xi_{[T]}) \in X_T, \quad \forall \,
    \xi_{[T]} \in \Xi_1 \times \ldots \times \Xi_T,
\end{array}
\right.
\end{equation}
where
$\xi_{[t]}=(\xi_1,\xi_2,\ldots,\xi_t)$.

This problem is also equivalent to the following Dynamic
Programming equations: 
\begin{equation}
\mathcal{Q}_0^* =
\displaystyle \min_{u_1 \in X_1} F_1(u_1)  + \mathcal Q_1 ( u_1 )
\end{equation}
where $\mathcal Q_1$ is obtained 
by a backward recursion as
follows:
set $\mathcal{Q}_T \equiv 0$
and from $t=T$ to $t=2$, set
\begin{equation}\label{secondstodp2-sto} 
 \mathfrak{Q}_{t-1}(\cdot \, ;\xi_t) = 
\displaystyle \min_{u_t \in X_t}  F_t (u_t,\cdot \, ;\xi_t) + \mathcal{Q}_{t}(u_t)
\end{equation}
and then define
\begin{equation}\label{def:Q-exp-stp}
    \mathcal{Q}_{t-1} (\cdot)= \mathbb{E}_{\xi_t}[ \mathfrak{Q}_{t-1}(\cdot \, ;\xi_t)].
\end{equation}

\section{BSDDP for $T$-stage strongly convex problems}\label{secsddp}

We now present the Bidirectional Stochastic Dual Dynamic Programming (BSDDP) algorithm tailored to solve the multistage stochastic problems with strongly convex costs presented in the previous section. The algorithm iteratively builds lower approximations of the true cost-to-go functions \( \mathcal{Q}_t \) using a collection of affine functions (cuts). These approximations are denoted by \( \Gamma_t^k \), and give rise to stagewise value functions \( V_{t-1}^k \), which in turn define the expected approximations \( \mathcal{V}_t^k \approx \mathcal{Q}_t \).

The method alternates between:
\begin{itemize}
    \item a forward pass, where a sample path is drawn and decisions are computed by minimizing stagewise models that use the current approximations;
    \item a backward pass, where cuts are added to refine the approximations based on subgradients of the sampled value functions;
    \item an update step, which introduces averaging to stabilize convergence and define an approximate first stage solution.
\end{itemize}

The definitions of \( \mathcal{V}_t^k \) and \( \Gamma_t^k \) will be central to the complexity analysis.
More precisely, 
\begin{equation}\label{defVtk}
\mathcal{V}_{t}^k (\cdot)= \mathbb{E}_{\xi_{t+1}}[V_{t}^k(\cdot\, ;\xi_{t+1})]
\end{equation}
where 
\begin{equation}\label{defVtk2}
V_{t}^k(x_{t}; \xi_{t+1}) := \min_{u_{t+1} \in X_{t+1}} F_{t+1}(u_{t+1}, x_{t}; \xi_{t+1}) + \Gamma_{t+1}^k(u_{t+1})
\end{equation}
for every $x_t \in X_t$.

BSDDP builds linearizations
of functions $\V_t^k$ for
$\Q_t$ at trial points computed in a forward pass.
These linearizations are built from linearizations of functions $V_{t}^k(\cdot; \xi_{t+1})$
at these trial points.
To compute  a linearization for these value functions 
$V_{t}^k(\cdot; \xi_{t+1})$,
we use Lemma 
\ref{subgradValueFunc1}-a),c) from the Appendix.
From this lemma, subgradients of  $V_{t}^k(\cdot; \xi_{t+1})$
at $x_t$ (used to build the linearizations)  are optimal Lagrange multipliers $\lambda_t$ for the constraint
$z_t=x_t$ in the reformulation
\begin{equation}\label{defVtk3}
V_{t}^k(x_{t}; \xi_{t+1})
=
\begin{aligned}
\min_{u_{t+1},\, z_t}\quad 
& F_{t+1}(u_{t+1}, z_t; \xi_{t+1}) + \Gamma_{t+1}^k(u_{t+1})\\
\text{s.t.}\quad 
& u_{t+1} \in X_{t+1}, z_t = x_t \;\;[\lambda_t].
\end{aligned}
\end{equation}
of problem \eqref{defVtk2}.
We will use for the linearizations of $\V_t^k$ at a point $x_t$ the notation
\begin{equation}
\ell_{\mathcal{V}_{t}^k}\left(\cdot \,; x_{t}\right) := 
\mathbb{E}_{\xi_{t+1}} \left[ V_{t}^k\left(x_{t}; \xi_{t+1}\right) 
+ \left\langle s_{t}^{k-1}(\xi_{t+1}), \cdot - x_{t} \right\rangle \right]
\end{equation}
where
\begin{equation}\label{subgradvtkx}
s_{t}^{k-1}(\xi_{t+1}) \in \partial V_{t}^k(\cdot \,; \xi_{t+1})\left(x_t \right).
\end{equation}

\begin{rem} \label{remsubbded}
By Assumption (A3), convexity of
$\Gamma_{t+1}^k$ (which is piecewise affine), and Lemma
\ref{subgradValueFunc2} from the Appendix, we have that $s_{t}^{k-1}(\xi_{t+1})$
in \eqref{subgradvtkx}
satisfies 
$\|s_{t}^{k-1}(\xi_{t+1})\| \leq M_t$ and that $V_t^k(\cdot,\xi_{t+1})$ is
$M_t$-Lipschitz continuous.
\end{rem}

We now state BSDDP.\\

\vspace{10pt}

\noindent\rule[0.5ex]{1\columnwidth}{1pt}
\par {\textbf{Bidirectional Stochastic Dual Dynamic Programming (BSDDP).}}\\
\noindent\rule[0.5ex]{1\columnwidth}{1pt}
\par {\textbf{Input.}}
$\tau_0 \in (0,1)$, cuts 
$\Gamma^1_t$ convex $M_t$-Lipschitz
such that $\Gamma_t^1 \le \mathcal{Q}_t$ for all $t = 1,\dots,T$, dictionary $\mathcal D$. Define $k=1$ and $\ell(1) = 1$.
Sample scenario $\zeta^1 = \left(\xi_2^1,\ldots,\xi_T^1 \right)$ and set $\mathcal{D}(\zeta^1)=1$.

\bigskip

\par {\textbf{Step 1. Forward pass.}} Let $x_1^k$ be the unique solution of
\[
\mathcal{V}_0^k := \min_{u_1 \in X_1} F_1(u_1) + \Gamma_1^k(u_1).
\]
For $t = 2$ to $T$, recursively define $x_t^k$ as the solution of
\[
V_{t-1}^k(x_{t-1}^k; \xi_t^k) := \min_{u_t \in X_t} F_t(u_t, x_{t-1}^k; \xi_t^k) + \Gamma_t^k(u_t).
\]

\par {\textbf{Step 2. Backward pass.}} Sample scenario $\zeta^{k+1} := \left(\xi_2^{k+1}, \ldots, \xi_T^{k+1} \right)$. 

\noindent If $\zeta^{k+1} \notin \mathrm{dom}(\mathcal{D})$,
\par $\hspace*{0.5cm}$set $\mathcal{D}(\zeta^{k+1}) := k+1$,
\par $\hspace*{0.5cm}$$\ell(k+1) = k+1$,
\par $\hspace*{0.5cm}$$\Gamma_{t}^{k+1} := \Gamma_{t}^k$ for all $t=1,\ldots,T$,
\par $\hspace*{0.5cm}$go to Step 3.

\noindent Else 
\par $\hspace*{0.5cm}$set $\ell(k+1) := \mathcal{D}(\zeta^{k+1})$,
\par $\hspace*{0.5cm}$define $\Gamma_T^{k+1} \equiv 0$.
\par $\hspace*{0.5cm}$For $t = T$ down to $t = 2$, do:
\par $\hspace*{1cm}$For each $\omega_t \in \Xi_t$, solve
\[
V_{t-1}^{k+1} \left(x_{t-1}^{\ell(k+1)}; \omega_t\right) := \min_{u_t \in X_t} F_t\left(u_t, x_{t-1}^{\ell(k+1)}; \omega_t\right) + \Gamma_t^{k+1}(u_t),
\]
\par $\hspace*{1.5cm}$and compute a subgradient
\[
s_{t-1}^k(\omega_t) \in \partial V_{t-1}^{k+1}(\cdot \,; \omega_t)\left(x_{t-1}^{\ell(k+1)}\right).
\]
\par $\hspace*{1cm}$End For
\par $\hspace*{1cm}$Construct the expected cutting-plane:
\[
\ell_{\mathcal{V}_{t-1}^{k+1}}\left(\cdot \,; x_{t-1}^{\ell(k+1)}\right) := 
\mathbb{E}_{\xi_t} \left[ V_{t-1}^{k+1}\left(x_{t-1}^{\ell(k+1)}; \xi_t\right) 
+ \left\langle s_{t-1}^k(\xi_t), \cdot - x_{t-1}^{\ell(k+1)} \right\rangle \right].
\]
\par $\hspace*{1cm}$Update the cut:
\begin{equation}\label{defGamma}
    \Gamma_{t-1}^{k+1} := \max\left\{ \Gamma_{t-1}^k, \ell_{\mathcal V_{t-1}^{k+1}}\left(\cdot \,; x_{t-1}^{\ell(k+1)}\right) \right\}.
\end{equation}
\par $\hspace*{0.5cm}$End For
\par $\hspace*{0.5cm}$Set $\mathcal{D}(\zeta^{k+1}) := k+1$.\\
\noindent End If

\bigskip

\par {\textbf{Step 3. Update of first stage solution.}} Compute $y_1^k$ such that
\[
y_1^k = (1 - \tau_0) x_1^k + \tau_0 y_1^{\ell(k)}
\]
Set \( k \leftarrow k+1 \), and return to Step 1.

\noindent\rule[0.5ex]{1\columnwidth}{1pt}

\par Some remarks are now in order. The following ingredients of BSDDP are common with SDDP:
\begin{itemize}
\item The forward pass of iteration $k$ computes trial points $x_t^k$ which are used in the backward passes to build cuts at these points.
    \item The functions
     $\mathcal{V}_{t}^k (\cdot)= \mathbb{E}_{\xi_{t+1}}[V_{t}^k(\cdot\, ;\xi_{t+1})]$
    serve as expected approximations of the true value functions \( \mathcal{Q}_t \), built from sample trajectories and the current approximations \( \Gamma_t^k \). They reflect the algorithm's evolving view of the stagewise cost-to-go functions.

    \item The condition \( \Gamma_t^k \le \mathcal{Q}_t \) is enforced at each iteration, ensuring that the algorithm maintains valid lower bounds. This monotonicity is central to the convergence of BSDDP.

    \item Subgradients in the backward pass are computed only at sampled points, making the method well-suited to high-dimensional problems and large scenario trees, avoiding the exponential complexity of exact dynamic programming.
\end{itemize}

\par The following ingredients are present in BSDDP but not in SDDP:
\begin{itemize}
    \item Whenever a scenario $\zeta^k$ appears for the first time in the algorithm, we have $\mathcal{D}(\zeta^k):=k$ and $\ell(k):=k$. In this case, $y_1^{\ell(k)}=y_1^k$, so the first stage update in Step 3 simplifies to \(y_1^k = (1-\tau_0)x_1^k + \tau_0 y_1^k,\) which reduces to $y_1^k = x_1^k$
    since $\tau_0<1$.
    
    \item In Step 2, if $\zeta^{k+1} \in \mathrm{dom}(\mathcal{D})$, the algorithm sets $\ell(k+1):=\mathcal{D}(\zeta^{k+1})$. By construction, $\ell(k+1)\leq k$ refers to the most recent iteration where $\zeta^{k+1}$ was sampled, and all states $x_t^{\ell(k+1)}$ were computed and stored at that time. Thus, $x_{t-1}^{\ell(k+1)}$ is available when accessed in the backward pass. If instead $\zeta^{k+1}\notin \mathrm{dom}(\mathcal{D})$, the algorithm sets $\ell(k+1):=k+1$ and assigns $\Gamma_t^{k+1}:=\Gamma_t^k$ for all $t$. No subgradients are computed, and $x_{t-1}^{\ell(k+1)}$ is never accessed in this case.    
    \item The update step \( y_1^k = (1 - \tau_0) x_1^k + \tau_0 y_1^{\ell(k)} \) introduces stabilization via averaging, akin to Polyak or momentum-based updates. This plays a critical role in obtaining the complexity bound proved later, for a well chosen value of $\tau_0$, see Theorem \ref{complexitysddp}.
    Note that the update step provides at iteration $k$ an approximate first stage solution $y_1^k$. Usually, MSOPs are solved in a rolling horizon setting (see e.g., \cite{rh1, rh2}) and therefore the first stage solution is often the only one which is implemented in practice.
\end{itemize}

To formalize the expectation step in the recursion, we define \( \mathcal{V}_{T}^k \equiv 0 \), and for each stage \( t \in \{1,\ldots,T-1\} \), recall that
\begin{equation}\label{def:V-exp-stp3}
    \mathcal{V}_{t}^k (\cdot)= \mathbb{E}_{\xi_{t+1}}[V_{t}^k(\cdot\, ;\xi_{t+1})].
\end{equation}

Finally, note that each decision \(x_t^k\) depends on the sampled history up to stage \(t\), summarized by:
\begin{equation}\label{dependence}
    \xi_{[t]}^{[k]} :=
\begin{cases}
\left( \xi_2^1,\xi_3^1,\ldots,\xi_T^1,\ldots,\xi_2^{k-1},\xi_3^{k-1},\ldots,\xi_T^{k-1} \right), & \text{if } t = 1, \\
\left( \xi_2^1,\xi_3^1,\ldots,\xi_T^1,\ldots,\xi_2^{k-1},\xi_3^{k-1},\ldots,\xi_T^{k-1},\xi_2^k,\ldots,\xi_t^k \right), & \text{if } t = 2, \dots, T.
\end{cases}
\end{equation}

\section{Complexity analysis of BSDDP applied to strongly convex problems}\label{secstronglyconvex}

In this section, we prove the complexity of the BSDDP algorithm 
stated in the previous section
when applied to multistage stochastic problems with strongly convex cost functions. Our goal is to quantify how the approximation errors evolve over the iterations and how they propagate across stages. To do this, we analyze the structure of the cut updates and their relationship to the true value functions \( \mathcal{Q}_t \).

The next lemma formalizes two basic properties of the algorithm’s approximations: monotonicity of the cut sequence and the quality of the cutting-plane approximation at sampled points. These properties serve as building blocks for bounding the overall error in later lemmas.

\begin{lemma}
Let \( t \in \{1,\ldots,T\} \) and 
$k$ satisfying
\( k > \ell(k) \ge 1 \). Then the cut sequence is monotone and bounded above by the true value function:
    \[
    \Gamma_t^{k-1} \le \Gamma_t^k \le \mathcal{Q}_t, \quad 
    \ell_{\mathcal{V}_t^k} \left(\cdot \,; x_t^{\ell(k)} \right) \le \mathcal{V}_t^k \le \mathcal{Q}_t.
    \]
\end{lemma}
\begin{proof}
The proof is by induction on $k$ and backward induction on $t$. For $t = T$, the relations trivially hold. If the relations hold for some $t \geq 2$ then we have
\begin{align*}
V_{t-1}^k\left(\cdot \,;\xi_{t}\right) &=
\displaystyle \min_{u_{t} \in X_{t}} F_{t}\left(u_{t},\cdot \,;\xi_{t}\right) + \Gamma_{t}^k\left(u_{t}\right) \\
&\le \min_{u_{t} \in X_{t}} F_{t}\left(u_{t},\cdot \,;\xi_{t}\right) + \mathcal Q_{t}\left(u_{t}\right)\\
&= \mathfrak{Q}_{t-1}\left(\cdot \,;\xi_{t}\right).
\end{align*}
It follows that for iterations $k$ where cuts are added
\begin{align*}
\ell_{\mathcal V_{t-1}^k}\left(\cdot \, ;x_{t-1}^{\ell(k)}\right) &= \mathbb{E}_{\xi_{t}} \left[ V_{t-1}^k\left(x_{t-1}^{\ell(k)};\xi_{t}\right)
+ \left\langle s_{t-1}^{k-1}\left(\xi_t\right), \cdot - x_{t-1}^{\ell(k)}  \right\rangle \right] \\
&\le \mathbb{E}_{\xi_{t}}\left[V_{t-1}^k\left(\cdot \,;\xi_{t}\right)\right] \\
&= \mathcal{V}_{t-1}^k\left(\cdot\right)\\
&\le \mathbb{E}_{\xi_{t}}\left[\mathfrak{Q}_{t-1}\left(\cdot \,;\xi_{t}\right)\right] \\
&= \mathcal Q_{t-1}\left(\cdot\right).
\end{align*}
Hence, we have
\[
\Gamma_{t-1}^{k-1}\left(\cdot\right) \le Q_{t-1}\left(\cdot\right),\;\; \ell_{\mathcal{V}_{t-1}^k}\left(\cdot \,;x_{t-1}^{\ell(k)}\right) \le \mathcal Q_{t-1}\left(\cdot\right),
\]
implying
\[
\Gamma_{t-1}^{k-1} \le \Gamma_{t-1}^k \le \mathcal Q_{t-1}.
\]
\end{proof}

We now establish regularity properties of the value functions $\mathcal{Q}_t$ and $\mathcal{V}_t^k$ used in the BSDDP algorithm. These functions are defined recursively through parametric optimization problems involving cost functions in the class $\mathcal{F}$.

The next lemma shows that both $\mathcal{Q}_t$ and its approximation $\mathcal{V}_t^k$ are convex and Lipschitz continuous, and that $\mathcal{V}_t^k$ admits affine minorants with controlled approximation error. These properties are fundamental for the validity of the cutting-plane construction and the complexity analysis that follows.

\begin{lemma}\label{linearizationBound}
For $t \in \{1,\ldots,T\}$ and $k \ge 1$, we have
\begin{enumerate}
    \item for every $y \in X_t$, linearization \(\ell_{\mathcal{V}_t^k}(\cdot \,;y)\) of \(\mathcal{V}_t^k\) satisfies
    \[
    \ell_{\mathcal{V}_t^k}(\cdot \,;y) \geq \mathcal V_t^k(y) - M_t\|\cdot - y\|;
    \]
    \item the function $\mathcal{Q}_t$ is convex and $M_t$-Lipschitz continuous;
    \item the function $\mathcal{V}_t^k$ is convex and $M_t$-Lipschitz continuous;
    \item the linearization from Item 1 also satisfies
    \begin{equation}\label{diffvflinapp}
    \mathcal{V}_{t}^k(\cdot) - \ell_{\mathcal{V}_{t}^k}(\cdot \,;y) \le 2 M_t \|\cdot-y\|.
    \end{equation}
\end{enumerate}
\end{lemma}

\begin{proof}
\begin{enumerate}
    \item The case $t = T$ is immediate. 
    We have
    \begin{align*}
        \ell_{\mathcal{V}_t^k}(\cdot \, ;y) &:= \mathbb{E}_{\xi_{t+1}} \left[ V_{t}^k\left(y; \xi_{t+1}\right) 
+ \left\langle s_{t}^{k-1}(\xi_{t+1}), \cdot - y \right\rangle \right] \\
        &= \mathcal{V}_t^k(y) + \left\langle \mathbb{E}_{\xi_{t+1}}[s_t^{k-1}(\xi_{t+1})], \cdot - y \right\rangle.
    \end{align*}
Using the fact that
$F_{t+1}(\cdot,\cdot \, ;\xi_{t+1}) \in {\cal F}(M_t,\mu_{t+1})$ and Remark \ref{remsubbded}, we obtain
\begin{equation}
\|s_t^{k-1}(\xi_{t+1})\| \leq M_t.
\end{equation}
    
    Then, for all $z \in X_t$,
    \begin{align*}
        \mathcal V_t^k(y) - \ell_{\mathcal{V}_t^k}(z;y) &\le \left\| \mathbb{E}_{\xi_{t+1}}[s_t^{k-1}(\xi_{t+1})] \right\| \cdot \|z - y\| \\
        &\le \mathbb{E}_{\xi_{t+1}}\left[ \|s_t^{k-1}(\xi_{t+1})\| \right] \cdot \|z - y\| \\
        &\le M_t \|z - y\|,
    \end{align*}
    proving the claim.

    \item The result follows by backward induction on $t$. Since $\mathcal{Q}_T \equiv 0$, the base case is trivial. Suppose $\mathcal{Q}_{t+1}$ is convex and $M_{t+1}$-Lipschitz continuous. For each $\omega_{t+1} \in \Xi_{t+1}$, $F_{t+1}(\cdot,\cdot \, ;\omega_{t+1}) \in {\cal F}(M_t,\mu_{t+1})$, and $\mathcal{Q}_{t+1} \in \overline{\text{Conv}}(X_{t+1})$. Then, by Lemma \ref{subgradValueFunc2},
    \[
    Q_t(\cdot ; \omega_{t+1}) = \min_{u_{t+1} \in X_{t+1}} F_{t+1}(u_{t+1},\cdot \, ; \omega_{t+1}) + \mathcal{Q}_{t+1}(u_{t+1})
    \]
    is convex and $M_t$-Lipschitz continuous. Therefore,
    \begin{align*}
        |\mathcal{Q}_t(x) - \mathcal{Q}_t(y)| &= \left| \mathbb{E}_{\xi_{t+1}}[Q_t(x;\xi_{t+1}) - Q_t(y;\xi_{t+1})] \right| \\
        &\le \mathbb{E}_{\xi_{t+1}} \left[ |Q_t(x;\xi_{t+1}) - Q_t(y;\xi_{t+1})| \right] \\
        &\le M_t \|x - y\|.
    \end{align*}

    \item The proof is identical to the proof  of Item 2, replacing $\mathcal{Q}_t$ with $\mathcal{V}_t^k$ and using the definition of $\mathcal{V}_t^k(\cdot)$ as an expectation of $V_t^k(\cdot;\xi_{t+1})$.

    \item We have  
    \begin{align*}
        \mathcal{V}_{t}^k(z) - \ell_{\mathcal{V}_{t}^k}(z \,;y) &\le \mathcal{V}_{t}^k(z) - \left[\mathcal{V}_t^k(y) - M_t\|z - y\|\right] \le 2 M_t \|z - y\|,
    \end{align*}
    where we have used the facts that
    $\V_t^k$ is $M_t$-Lipschitz and by Item 1,
    the norms
    of subgradients of $\mathcal{V}_t^k$ are bounded above by $M_t$. This
    completes the proof.
\end{enumerate}
\end{proof}

\vspace{0.2cm}

We now introduce key quantities that play a central role in the complexity analysis of the BSDDP algorithm. More precisely, setting
\begin{equation}\label{Vtildetk}
\tilde V_t^k = V_t^k \left(x_t^k; \xi_{t+1}^k \right) \mbox{ and }
\tilde Q_t^k = \mathfrak{Q}_t\left(x_t^k; \xi_{t+1}^k \right),
\end{equation}
we define the approximation gap between the true value function $\mathcal{Q}_t$ and its $k$-th iteration approximation $\mathcal{V}_t^k$, evaluated at the forward pass iterate $x_t^k$:
\begin{equation}\label{etkdef}
e_{t}^k = \left(\mathcal{Q}_{t} - \mathcal{V}_{t}^k \right)\left(x_{t}^k \right), \quad
\tilde e_t^k = \tilde Q_t^k - \tilde V_t^k,t=1,\ldots,T-1.
\end{equation}

 We also define the distance between decision variable iterates at stage $t$:

\begin{equation}\label{defdeltatk}
\Delta_t^{k,\ell} = \left\|x_{t}^k-x_{t}^\ell\right\|,t=1,\ldots,T-1.
\end{equation}

We finally define the event
\begin{equation}\label{defEventEk}
E^k = \Big\{\text{for all} \quad t=1, \ldots, T-1: \quad e_t^k \le \tilde e_t^k \Big\}.
\end{equation}

The next result establishes an inequality relating the approximation error at stage $t-1$ to that at stage $t$. This recursive structure is essential to our analysis, as it enables the propagation and control of errors across stages.

\begin{lemma}\label{QminusVrecursion}
For every stage \( t \in \{1,\ldots,T\} \) and iteration \( k \ge 1 \), the following hold:
\begin{enumerate}
    \item
    \begin{equation}\label{QminusVxi}
            \tilde e_t^k
    \le [\mathcal Q_{t+1} - \Gamma_{t+1}^k]\left(x_{t+1}^k \right);
    \end{equation}
    \item we have $\tilde e_{T-1}^k = e_{T-1}^k = 0$ and for $\ell=\ell(k)<k$ and $t<T-1$, 
            \begin{equation}\label{etBoundDelta0}
        \tilde e_t^k \le e_{t+1}^k + 2M_{t+1}\Delta_{t+1}^{k,\ell};
        \end{equation}
    \item if \eqref{defEventEk} holds and $\ell=\ell(k)<k$, then for $t<T-1$,
    \begin{equation}\label{etBoundDelta}
        e_t^k \le 2\sum_{j=t+1}^{T-1} M_{j}\Delta_{j}^{k,\ell}. 
    \end{equation}
\end{enumerate}
\end{lemma}

\begin{proof}
\begin{enumerate}
    \item By definition of \( V_{t}^k \), we have
    \begin{align*}
        V_t^k\left(x_t^k; \xi_{t+1}^k \right) 
        &= F_{t+1}\left(x_{t+1}^k, x_t^k; \xi_{t+1}^k \right) + \Gamma_{t+1}^k\left(x_{t+1}^k \right) \\
        &= \left[ F_{t+1}\left(x_{t+1}^k, x_t^k; \xi_{t+1}^k \right) + \mathcal{Q}_{t+1}\left(x_{t+1}^k \right) \right] 
        - \left[ \mathcal Q_{t+1}\left(x_{t+1}^k \right) - \Gamma_{t+1}^k\left(x_{t+1}^k \right) \right] \\
        &\ge \mathfrak{Q}_t \left(x_t^k ; \xi_{t+1}^k \right) - \left[ \mathcal Q_{t+1}\left(x_{t+1}^k \right) - \Gamma_{t+1}^k\left(x_{t+1}^k \right) \right],
    \end{align*}
    where the last inequality follows from the definition of \( \mathfrak{Q}_t \).
    \item The identity $\tilde e_{T-1}^k = e_{T-1}^k = 0$ follows from the fact that $\mathfrak{Q}_T(\cdot, \omega_T) \equiv V_T^k(\cdot, \omega_T) \equiv 0$ for all $k$, \(\omega_T \in \Xi_T\). Moreover, for $\ell=\ell(k)<k$ and $t<T-1$, we have
    \begin{align*}
        \tilde e_{t}^k &\hspace{-0.3em}\stackrel{\eqref{QminusVxi}}{\le} \left(\mathcal{Q}_{t+1} - \Gamma_{t+1}^k \right)\left(x_{t+1}^k \right) \\
        &= \left(\mathcal{Q}_{t+1} - \mathcal{V}_{t+1}^k \right)\left(x_{t+1}^k \right) + \left(\mathcal{V}_{t+1}^k - \Gamma_{t+1}^k \right)\left(x_{t+1}^k \right) \\
        &= e_{t+1}^k + \left(\mathcal{V}_{t+1}^k - \Gamma_{t+1}^k \right)\left(x_{t+1}^k \right) \\
        &\hspace{-0.1em}\stackrel{\eqref{defGamma}}{\le} e_{t+1}^k + \mathcal{V}_{t+1}^k\left(x_{t+1}^k \right) - \ell_{\mathcal{V}_{t+1}^k}\left(x_{t+1}^k; x_{t+1}^\ell \right) \\
        &\hspace{-0.3em}\stackrel{\eqref{diffvflinapp}}{\le} e_{t+1}^k + 2 M_{t+1} \left\|x_{t+1}^k - x_{t+1}^\ell \right\| \\
        &= e_{t+1}^k + 2 M_{t+1} \Delta_{t+1}^{k,\ell}.
    \end{align*}
    \item If \eqref{defEventEk} holds
    then 
    $e_j^k \leq \tilde e_j^k \leq e_{j+1}^k+2M_{j+1} \Delta_{j+1}^{k, \ell}$ for 
    $j=t,\ldots,T-2$.
    The result then follows by summing these inequalities from  $j=t$ up to $T-2$ and using $e_{T-1}^k=0$.
\end{enumerate}
\end{proof}

\vspace*{0.5cm}
\par We now analyze how the value function $\mathcal{V}_0^k$ compares to a previous iterate $\mathcal{V}_0^\ell$. The lemma below establishes a lower bound on the difference $\mathcal{V}_0^k - \tau \mathcal{V}_0^\ell$ for any convex combination parameter $\tau \in [0,1]$. This result is key to understanding how strong convexity translates into measurable descent in the BSDDP recursion and ultimately leads to complexity guarantees.

\begin{lemma}\label{V0recursion}
For $k > \ell=\ell(k) \ge 1$,
and every $\tau \in [0,1]$, we have:
\begin{align*}
    \mathcal{V}_0^k - \tau \mathcal{V}_0^\ell \ge 
    (1 - \tau) (F_1 + \mathcal{Q}_1)\left(x_1^k \right) 
    - (1 - \tau)\left[ e_1^k + 2 M_1 \Delta_1^{k,\ell} \right] 
    + \frac{\mu_1 \tau}{2} \left(\Delta_1^{k,\ell} \right)^2.
    \end{align*}
\end{lemma}

\begin{proof} Since $x_1^\ell$ is the minimizer of a $\mu_1$-strongly convex function $\left(F_1 + \Gamma_1^\ell \right)$ over $X_1$, we have
    \begin{equation*}
    \left(F_1 + \Gamma_1^\ell \right)(u_1) \ge \left(F_1 + \Gamma_1^\ell \right) \left(x_1^\ell \right) + \frac{\mu_1}{2} \left\|u_1 - x_1^\ell \right\|^2, \quad \forall u_1 \in X_1.
    \end{equation*}
    Applying this inequality at $u_1 = x_1^k$ yields
    \begin{align}\label{V0strongconv}
    \left(F_1 + \Gamma_1^\ell \right)\left(x_1^k \right) \ge \mathcal{V}_0^\ell + \frac{\mu_1}{2} \left\|x_1^k - x_1^\ell\right\|^2.
    \end{align}
    On the other hand, $\ell <k$ which implies by monotonicity of models $\Gamma_1^k$ that $\Gamma_1^k \geq \Gamma_1^{\ell}$.
    Then for the value function at iteration $k$, we have 
    \begin{align*}
    \mathcal{V}_0^k &= F_1\left(x_1^k \right) + \Gamma_1^k\left(x_1^k \right) \\
    &\hspace{-0.1em}\stackrel{\eqref{defGamma}}{\ge} F_1\left(x_1^k \right) + \left[ (1 - \tau)\ell_{\mathcal{V}_1^k}\left(x_1^k; x_1^\ell \right) + \tau \Gamma_1^\ell\left(x_1^k \right) \right] \\
    &= (1 - \tau)\left[ F_1\left(x_1^k \right) + \ell_{\mathcal{V}_1^k}\left(x_1^k; x_1^\ell \right) \right] + \tau \left(F_1 + \Gamma_1^\ell \right)\left(x_1^k \right) \\
    &\hspace{-0.3em}\stackrel{\eqref{V0strongconv}}{\ge} (1 - \tau)\left[ F_1\left(x_1^k \right) + \ell_{\mathcal{V}_1^k}\left(x_1^k; x_1^\ell \right) \right] + \tau \left[ \mathcal{V}_0^\ell + \frac{\mu_1}{2} \left\|x_1^k - x_1^\ell\right\|^2 \right].
    \end{align*}
    Subtracting $\tau \mathcal{V}_0^\ell$ from both sides gives
    \begin{align*}
    \mathcal{V}_0^k - \tau \mathcal{V}_0^\ell &\ge 
    (1 - \tau)\left[ F_1\left(x_1^k \right) + \ell_{\mathcal{V}_1^k}\left(x_1^k; x_1^\ell \right) \right] + \frac{\tau \mu_1}{2} \left\|x_1^k - x_1^\ell\right\|^2 \\
    &\hspace{-0.4em}\stackrel{\eqref{diffvflinapp}}{\ge} (1 - \tau)\left[ (F_1 + \mathcal{V}_1^k)\left(x_1^k \right) - 2 M_1 \left\|x_1^k - x_1^\ell\right\| \right] + \frac{\tau \mu_1}{2} \left\|x_1^k - x_1^\ell\right\|^2.
    \end{align*}
    We obtain
    \begin{align*}
    \mathcal{V}_0^k - \tau \mathcal{V}_0^\ell &\ge (1 - \tau) \left[ (F_1 + \mathcal{V}_1^k)\left(x_1^k \right) - 2 M_1 \left\|x_1^k - x_1^\ell\right\| \right] + \frac{\tau \mu_1}{2} \left\|x_1^k - x_1^\ell\right\|^2\\
    &= (1 - \tau)\left[(F_1 + \mathcal{Q}_1)\left(x_1^k \right) - (\mathcal{Q}_1 - \mathcal{V}_1^k)\left(x_1^k \right) - 2 M_1 \Delta_1^{k,\ell} \right] + \frac{\mu_1 \tau}{2} \left(\Delta_1^{k,\ell} \right)^2 \\
    &= (1 - \tau)\left[ (F_1 + \mathcal{Q}_1)\left(x_1^k \right) - e_1^k - 2 M_1 \Delta_1^{k,\ell} \right] + \frac{\mu_1 \tau}{2} \left(\Delta_1^{k,\ell} \right)^2,
    \end{align*}
    completing the proof.
\end{proof}

\vspace*{0.5cm}

We now consider the expected optimality gap between the true value $(F_1 + \mathcal{Q}_1)\left(y_1^k \right)$, evaluated at the convex combination point $y_1^k$, and the approximated value function $\mathcal{V}_0^k$. This quantity is denoted by $e_0^k$ and serves as a global measure of error for the BSDDP algorithm. The lemma below provides a recursive inequality for $e_0^k$ that captures its evolution across iterations where
\[
e_0^k = \left(F_1 + \mathcal{Q}_1\right)\left(y_1^k \right) - \mathcal{V}_0^k.
\]

\begin{lemma}\label{lembde0}
For $k> \ell=\ell(k) \ge 1$, we have
\begin{equation}\label{first0cruc}
e_0^k - \tau_0 e_0^\ell
\leq (1 - \tau_0) \left[ e_1^k + 2 M_1 \Delta_1^{k,\ell} \right] 
- \frac{\tau_0 \mu_1}{2} \left(\Delta_1^{k,\ell} \right)^2,
\end{equation}
where $\tau_0$ is the averaging parameter used in BSDDP.
\end{lemma}

\begin{proof}
By convexity of the function $F_1 + \mathcal{Q}_1$ and using the definition of $y_1^k$, we have
\begin{equation}\label{F1plusQ1conv}
    \begin{aligned}
        (F_1 + \mathcal{Q}_1)\left(y_1^k \right) &= (F_1 + \mathcal{Q}_1) \left((1 - \tau_0)x_1^k + \tau_0 y_1^\ell \right)  
\\&\le (1 - \tau_0)(F_1 + \mathcal{Q}_1)\left(x_1^k \right) + \tau_0 (F_1 + \mathcal{Q}_1) \left(y_1^\ell \right).
    \end{aligned}
\end{equation}
Hence,
\begin{align*}
e_0^k - \tau_0 e_0^\ell
&= (F_1 + \mathcal{Q}_1)\left(y_1^k \right) - \mathcal{V}_0^k
- \tau_0\left[(F_1 + \mathcal{Q}_1)\left(y_1^\ell \right) - \mathcal{V}_0^\ell \right]\\
&= (F_1 + \mathcal{Q}_1)\left(y_1^k \right) - \tau_0 (F_1 + \mathcal{Q}_1)\left(y_1^\ell \right)
- \left[\mathcal{V}_0^k - \tau_0 \mathcal{V}_0^\ell \right]\\
&\hspace{-0.3em}\stackrel{\eqref{F1plusQ1conv}}{\le} (1 - \tau_0) (F_1 + \mathcal{Q}_1)\left(x_1^k \right)
- \left[ \mathcal{V}_0^k - \tau_0 \mathcal{V}_0^\ell \right].
\end{align*}
Applying the bound from Lemma \ref{V0recursion} written for $\tau=\tau_0$, which provides a lower bound on $\mathcal{V}_0^k - \tau_0 \mathcal{V}_0^\ell$, we obtain:
\begin{align*}
&e_0^k - \tau_0 e_0^\ell - (1 - \tau_0) (F_1 + \mathcal{Q}_1)\left(x_1^k \right)\\
&\le - \left( (1 - \tau_0) \left[(F_1 + \mathcal{Q}_1)\left(x_1^k \right) \right] - (1 - \tau_0) \left[e_1^k + 2 M_1 \Delta_1^{k,\ell} \right] + \frac{\tau_0 \mu_1}{2} \left(\Delta_1^{k,\ell} \right)^2 \right),
\end{align*}
easily implying the desired result.
\end{proof}

\vspace*{0.5cm}

We now analyze the behavior of the approximated value functions $\mathcal{V}_t^k$ under the averaging step of the BSDDP algorithm. Specifically, the following lemma establishes a lower bound on the difference $\tilde V_t^k - \tau \tilde V_t^\ell$, which plays a key role in the recursive control of the global error terms.

\begin{lemma}\label{lemlbvtk}
For  $t \in \{1,\ldots,T-1\}$, $k > \ell =\ell(k) \ge 1$ and every \(\tau \in [0,1]\),
\[
    \tilde V_t^k - \tau \tilde V_t^\ell \ge \tau \left( \frac{\mu_{t+1}}{2}\left(\Delta_{t+1}^{k,\ell} \right)^2 - M_t \Delta_t^{k,\ell} \right) + (1-\tau) \Big( \tilde Q_t^k - e_{t+1}^k - 2M_{t+1}\Delta_{t+1}^{k,\ell} \Big).
\]
\end{lemma} 
\begin{proof} Using strong convexity of the objective
function 
$F_{t+1} \left(\cdot, x_{t}^\ell; \xi_{t+1}^\ell \right) + \Gamma_{t+1}^\ell\left( \cdot \right)$
minimized over $X_{t+1}$ with minimum $x_{t+1}^{\ell}$, we have
for $x_{t+1}^k \in X_{t+1}$ that
\begin{equation}\label{sctp1f}
F_{t+1} \left(x_{t+1}^k, x_{t}^\ell; \xi_{t+1}^\ell \right) + \Gamma_{t+1}^\ell\left(x_{t+1}^k \right) 
\ge F_{t+1} \left(x_{t+1}^\ell, x_{t}^\ell; \xi_{t+1}^\ell \right) + \Gamma_{t+1}^\ell \left(x_{t+1}^\ell \right) + \frac{\mu_{t+1}}{2} \left\|x_{t+1}^k - x_{t+1}^\ell \right\|^2.
\end{equation}

Notice that $\ell=\ell(k)$ implies $\xi_{t+1}^k = \xi_{t+1}^\ell$. Hence
\begin{align*}
    V_{t}^k\left(x_{t}^k; \xi_{t+1}^k \right)   
    &= F_{t+1}\left(x_{t+1}^k, x_{t}^k; \xi_{t+1}^k \right) + \Gamma_{t+1}^k\left(x_{t+1}^k \right) 
    \\&\ge F_{t+1} \left(x_{t+1}^k, x_{t}^k; \xi_{t+1}^\ell \right) + \Gamma_{t+1}^\ell\left(x_{t+1}^k \right) 
    \\&\ge F_{t+1} \left(x_{t+1}^k, x_{t}^\ell; \xi_{t+1}^\ell \right) + \Gamma_{t+1}^\ell\left(x_{t+1}^k \right) - M_t\left\|x_{t}^k - x_{t}^\ell \right\|
    \\&  \stackrel{\eqref{sctp1f}}{\ge} F_{t+1} \left(x_{t+1}^\ell, x_{t}^\ell; \xi_{t+1}^\ell \right) + \Gamma_{t+1}^\ell \left(x_{t+1}^\ell \right) + \frac{\mu_{t+1}}{2} \left\|x_{t+1}^k - x_{t+1}^\ell \right\|^2 - M_t\left\|x_{t}^k - x_{t}^\ell \right\| \\
    &= V_{t}^{\ell} \left(x_{t}^{\ell}; \xi_{t+1}^\ell \right) + \frac{\mu_{t+1}}{2} \left\|x_{t+1}^k - x_{t+1}^\ell \right\|^2 - M_t\left\|x_{t}^k - x_{t}^\ell \right\|,
\end{align*}
that is
\begin{equation}\label{lemlbvtkP1}
    \tilde V_t^k \ge \tilde V_t^\ell + \frac{\mu_{t+1}}{2}\left(\Delta_{t+1}^{k,\ell} \right)^2 - M_t \Delta_t^{k,\ell}.
\end{equation}
Now, since $\ell=\ell(k) < k,$
\begin{align*}
    V_{t}^k\left(x_{t}^k; \xi_{t+1}^k \right) &= F_{t+1}\left(x_{t+1}^k, x_{t}^k; \xi_{t+1}^k \right) + \Gamma_{t+1}^k\left(x_{t+1}^k \right) 
    \\&\ge F_{t+1}\left(x_{t+1}^k, x_{t}^k; \xi_{t+1}^k \right) + \ell_{\mathcal V_{t+1}^k} \left(x_{t+1}^k; x_{t+1}^\ell \right)
    \\&\ge F_{t+1}\left(x_{t+1}^k, x_{t}^k; \xi_{t+1}^k \right) + \mathcal V_{t+1}^k\left(x_{t+1}^k \right) - 2M_{t+1} \left\|x_{t+1}^k - x_{t+1}^\ell\right\|
    \\&= F_{t+1}\left(x_{t+1}^k, x_{t}^k; \xi_{t+1}^k \right) + \mathcal Q_{t+1}\left(x_{t+1}^k \right) - \left(\mathcal Q_{t+1}-\mathcal V_{t+1}^k \right)\left(x_{t+1}^k \right) - 2M_{t+1} \left\|x_{t+1}^k - x_{t+1}^\ell\right\|
    \\&\ge \mathfrak{Q}_{t}\left(x_{t}^k; \xi_{t+1}^k \right) - \left(\mathcal Q_{t+1}-\mathcal V_{t+1}^k \right)\left(x_{t+1}^k \right) - 2M_{t+1} \left\|x_{t+1}^k - x_{t+1}^\ell\right\|,
\end{align*}
or
\begin{equation}\label{lemlbvtkP2}
    \tilde V_t^k \ge \tilde Q_t^k - e_{t+1}^k - 2M_{t+1}\Delta_{t+1}^{k,\ell}.
\end{equation}
Uniting \eqref{lemlbvtkP1} and \eqref{lemlbvtkP2}, we get
    \begin{equation}
        \begin{aligned}
            \tilde V_t^k &\ge \tau \left( \tilde V_t^\ell + \frac{\mu_{t+1}}{2}\left(\Delta_{t+1}^{k,\ell} \right)^2 - M_t \Delta_t^{k,\ell} \right) + (1-\tau) \Big( \tilde Q_t^k - e_{t+1}^k - 2M_{t+1}\Delta_{t+1}^{k,\ell} \Big).
        \end{aligned}
    \end{equation}
The desired result comes from rewriting the inequality above.
\end{proof}

\vspace*{0.5cm}

In the next lemma, we bound the error propagation at stage $t$ in terms of the errors at stages $t$ and $t+1$, as well as the progress quantities $\Delta_t^{k,\ell}$ and $\Delta_{t+1}^{k,\ell}$. This is crucial for establishing a recursive inequality on the error terms that will later support the complexity analysis.

\begin{lemma}\label{lembdet}
For $t \in \{1,\ldots,T-1\}$, $k > \ell = \ell(k) \ge 1$ and every \(\tau \in [0,1]\),
\begin{equation}\label{first1cruc}
    \tilde e_t^k - \tau \tilde e_t^\ell \le \tau \left( 2M_t \Delta_t^{k,\ell} - \frac{\mu_{t+1}}{2}\left(\Delta_{t+1}^{k,\ell} \right)^2 \right) + (1-\tau) \Big( e_{t+1}^k + 2M_{t+1}\Delta_{t+1}^{k,\ell} \Big).
\end{equation}
\end{lemma}

\begin{proof} 
The inequality easily follows from
        \begin{equation}
        \begin{aligned}
            \tilde e_t^k - \tau \tilde e_t^\ell &= \tilde Q_t^k - \tilde V_t^k - \tau \tilde Q_t^\ell + \tau \tilde V_t^\ell
            \\ &\le  \tilde Q_t^k - \tau \tilde Q_t^\ell - \tau \left( \frac{\mu_{t+1}}{2}\left(\Delta_{t+1}^{k,\ell} \right)^2 - M_t \Delta_t^{k,\ell} \right) - (1-\tau) \Big( \tilde Q_t^k - e_{t+1}^k - 2M_{t+1}\Delta_{t+1}^{k,\ell} \Big),
        \end{aligned}
    \end{equation}
    the relation
$\mathfrak{Q}_t \left(\cdot, \xi_{t+1}^k \right) \equiv \mathfrak{Q}_t \left(\cdot, \xi_{t+1}^\ell \right)$
which holds for
$\ell=\ell(k)$
and $M_t$-Lipschitz continuity of 
$\mathfrak{Q}_t \left(\cdot, \xi_{t+1}^k \right) \equiv \mathfrak{Q}_t \left(\cdot, \xi_{t+1}^\ell \right)$.
\end{proof}

\vspace*{0.5cm}

\begin{rem} A crucial fact
in our analysis, which uses strong convexity, are the terms $-\frac{\tau \mu_{t+1}}{2}(\Delta_{t+1}^{k,\ell})^2$ in the right-hand sides of inequalities \eqref{first0cruc} and \eqref{first1cruc}.
Considering a weighted sum of these inequalities, for well-chosen weights, will provide us terms
of form 
$-a_t x_t^2+b_t x_t$ for $x_t=\Delta_t^k$
and $a_t=\frac{\tau}{2}\mu_t \theta_{t-1}>0$ for some weight $\theta_{t-1}>0$.
This will allow us to eliminate
terms $\Delta_t^k$ using the inequality
$-a_t x_t^2+b_t x_t \leq b_t^2/4a_t$, see Lemma \ref{prooflemreccruc} below and its proof for details.
\end{rem}

\vspace*{0.5cm}

On the basis of our previous remark, we now introduce a bound on a weighted aggregate error measure and derive a recursive bound that will be instrumental in proving the iteration complexity of BSDDP method.

\begin{lemma}\label{prooflemreccruc} Let \( \varepsilon \in (0, 1/2) \) be given. Define
\[
\tau := \frac{1}{1 + \varepsilon^{2^{T-1} - 1}} \in (0,1), \qquad 
\theta_t := \varepsilon^{2^{T-1} - 2^{T-1 - t}}, \quad t = 0, \ldots, T-2.
\]
Asssume that nonnegative constants \(w_t^k, \Delta_t^k\) satisfy for $t=0,\ldots,T-2$:
\[
w_t^{k+1} - \tau w_t^k 
\le 2(1 - \tau) \sum_{j=t+1}^{T-1} M_j \Delta_j^k + 2 \tau M_t \Delta_t^k - \frac{\tau \mu_{t+1}}{2} (\Delta_{t+1}^k)^2,
\]
with $\Delta_0^k = 0.$ Finally, let $r^k$ be the weighted sum
\[
r^k := \sum_{t=0}^{T-2} \theta_t w_t^k.
\]
Then the recursive inequality
\[
r^{k+1} \le \tau r^k + (1 - \tau) \sum_{t=1}^{T-1} \frac{20 M_t^2 \varepsilon}{\mu_t}
\]
holds.
\end{lemma}

\begin{proof}
Multiplying each inequality by $\theta_t$ and summing from $t = 0$ to $T-2$, we get:
\begin{align*}
&r^{k+1} - \tau r^k 
\\&\le 2(1 - \tau) \left[ \sum_{j=1}^{T-1} M_j \Delta_j^k + \sum_{t=1}^{T-2} \theta_t \sum_{j=t+1}^{T-1} M_j \Delta_j^k \right] + 2\tau \sum_{t=1}^{T-2} \theta_t M_t \Delta_t^k - \frac{\tau}{2} \sum_{t=0}^{T-2} \theta_t \mu_{t+1} (\Delta_{t+1}^k)^2.
\end{align*}
Note that
\[
\sum_{t=1}^{T-2} \theta_t \sum_{j=t+1}^{T-1} M_j \Delta_j^k = \sum_{j=2}^{T-1} M_j \Delta_j^k \sum_{t=1}^{j-1} \theta_t.
\]
Thus, we can rewrite the bound as
\begin{align*}
r^{k+1} - \tau r^k + \frac{\tau}{2} \sum_{t=1}^{T-1} \theta_{t-1} \mu_t (\Delta_t^k)^2 
&\le 2(1 - \tau) \left[ \sum_{j=1}^{T-1} M_j \Delta_j^k + \sum_{j=2}^{T-1} M_j \Delta_j^k \sum_{t=1}^{j-1} \theta_t \right] + 2\tau \sum_{j=1}^{T-1} M_j \Delta_j^k \theta_j \\
&= 2 \sum_{t=1}^{T-1} M_t \Delta_t^k \left[\tau \theta_t + (1 - \tau)\sum_{j=0}^{t-1} \theta_j\right].
\end{align*}
Finally, applying the inequality $-a_t x_t^2 + b_t x_t \le \frac{b_t^2}{4a_t}$ with
\[
a_t = \frac{\tau}{2} \theta_{t-1} \mu_t, \quad b_t = 2 M_t \left[ \tau \theta_t + (1 - \tau) \sum_{j=0}^{t-1} \theta_j \right], \quad x_t = \Delta_t^k,
\]
gives
\[
2 M_t  \left[\tau \theta_t + (1 - \tau)\sum_{j=0}^{t-1} \theta_j\right] \Delta_t^k - \frac{\tau}{2} \theta_{t-1} \mu_t (\Delta_t^k)^2 \le (1-\tau) C_t,
\]
where
\[
C_t := \displaystyle \frac{2 M_t^2 \left[\tau \theta_t +  (1-\tau)\displaystyle \sum_{j=0}^{t-1} \theta_j\right]^2}{\tau (1-\tau)\mu_t \theta_{t-1}}, \quad \text{for } \theta_0 := 1 \text{ and } \theta_{T-1} := 0.
\]
Now, we observe that
\[
2^{T-2} - 2^{T-1 - t} = 2^{T-1 - t}(2^{t-1} - 1) \ge 2^{t-1} - 1 \ge t - 1.
\]
This ensures that the exponent defining \( \theta_t \) grows with \( t \), and in particular, for all \( t \in \{1, \ldots, T-2\} \), we have:
\[
\frac{1 - \tau}{\tau \theta_{t-1}} = \varepsilon^{2^{T - t} - 1} \le \varepsilon, \qquad
\frac{\tau \theta_t^2}{(1 - \tau) \theta_{t-1}} = \varepsilon.
\]

Next, consider the term \( \displaystyle \sum_{j = 0}^{t-1} \theta_j \). Using the definition of \( \theta_t \), we can bound this term as:
\[
\sum_{j = 0}^{t-1} \theta_j 
= 1 + \sum_{j = 1}^{t-1} \varepsilon^{2^{T-1} - 2^{T - 1 - j}} 
\le 1 + \varepsilon^{2^{T - 2}} \sum_{j = 1}^{t-1} \varepsilon^{j - 1} 
\le 1 + \frac{\varepsilon^{2^{T - 2}}}{1 - \varepsilon} 
\le 1 + \frac{\varepsilon}{1 - \varepsilon} 
\le 2,
\]
where the last inequality uses \( \varepsilon < 1/2 \).

We now bound the constant \( C_t \) defined in the previous lemma:
\[
C_t 
\le \frac{2 M_t^2 \left[ 2(1 - \tau) + \tau \theta_t \right]^2}{\mu_t \tau \theta_{t-1} (1 - \tau)}.
\]
Using \( (a + b)^2 \le 2a^2 + 2b^2 \) and the bounds above:
\[
C_t 
\le \frac{4 M_t^2}{\mu_t} \left( \frac{4(1 - \tau)}{\tau \theta_{t-1}} + \frac{\tau \theta_t^2}{(1 - \tau) \theta_{t-1}} \right) 
\le \frac{20 M_t^2 \varepsilon}{\mu_t}.
\]
We finally obtain
$$
r^{k+1}-\tau r^k \leq (1-\tau)\sum_{t=1}^{T-1} C_t \leq 
 (1 - \tau) \sum_{t=1}^{T-1} \frac{20 M_t^2 \varepsilon}{\mu_t},
$$
which achieves the proof of the lemma.
\end{proof}

\vspace{0.2cm}

To control the error $e_0^k$ at a future iteration $k$, we leverage the recursive structure of the error bounds together with a careful choice of scenario sequences. In particular, if we collect enough iterations where the events \( E^k \) occur and the sampled scenarios coincide, then we can show that the error at the first stage decays exponentially fast in the number of such iterations. This is formalized in the following lemma.

Here, we introduce the constant 
\begin{equation}\label{defc}
c = 2 \sum_{t=1}^{T-1} M_t D_t,
\end{equation}
where $D_t := \sup_{x,x' \in X_t} \|x - x'\| < \infty$ is the diameter of $X_t.$

\begin{lemma}\label{lemPHP}
    Suppose that $k_1<k_2<\ldots < k_{P+1}$ satisfy, for all $i=2,\ldots,P,P+1$: event $E^{k_i}$ occurs and \(k_{i-1} = \ell(k_{i}) < k_i.\) Then, for
    \begin{equation}\label{defTau}
    \varepsilon = \frac{\bar \varepsilon}{\displaystyle 40 \sum_{t=1}^{T-1} \frac{M_t^2}{\mu_t}} < \frac{1}{2}, \quad \tau = \frac{1}{1 + \varepsilon^{2^{T-1} - 1}},
    \end{equation}
    we have
    \[
    e_0^{k_{P+1}} \le \frac{\bar \varepsilon}{2} + 2c \cdot \tau^{P}.
    \]
\end{lemma}

\begin{proof}
Since event $E^{k_i}$ holds, we have
$e_t^{k_{i}} \le \tilde e_t^{k_{i}}$
and using inequalities~\eqref{etBoundDelta0} and \eqref{etBoundDelta}, we have for each stage $t \ge 1$ and each index $i=1,\ldots,P$:
\begin{equation}\label{lemPHPineq1}
    e_t^{k_{i}} \le \tilde e_t^{k_{i}} \le 2 \sum_{j=t+1}^{T-1} M_j \Delta_j^{k_{i+1},k_i}.
\end{equation}
Using \eqref{lemPHPineq1} and
Lemma~\ref{lembdet}, we have for each stage $t=1,\ldots,T-2$, and each index $i=1,\ldots,P$:
    \[
    \tilde e_t^{k_{i+1}} - \tau \tilde e_t^{k_i} 
    \le \tau \left( 2M_t \Delta_t^{k_{i+1},k_i} - \frac{\mu_{t+1}}{2} \left(\Delta_{t+1}^{k_{i+1},k_i} \right)^2 \right) 
    + 2(1-\tau) \sum_{j=t+1}^{T-1} M_j \Delta_j^{k_{i+1},k_i}.
    \]
    Using Lemma
    \ref{lembde0},
    we also have for $i=1,\ldots,P$,
    \[
    e_0^{k_{i+1}} - \tau e_0^{k_i} 
    \le - \tau \frac{\mu_{1}}{2} \left(\Delta_{1}^{k_{i+1},k_i} \right)^2 
    + 2(1-\tau) \sum_{j=1}^{T-1} M_j \Delta_j^{k_{i+1},k_i}.
    \]
    For brevity, let us denote for $i=1,\ldots,P$,
    \[
w_0^i = e_0^{k_i},\quad 
    w_t^i := \tilde e_t^{k_i},\;\;t=1,\ldots,T-2, \quad \Delta_t^i := \Delta_t^{k_{i+1},k_i},\;\;t=1,\ldots,T-1.
    \]
    Then the recurrences simplify to:
    \[
    w_t^{i+1} - \tau w_t^i 
    \le \tau \left( 2M_t \Delta_t^i - \frac{\mu_{t+1}}{2}(\Delta_{t+1}^i)^2 \right) 
    + 2(1-\tau) \sum_{j=t+1}^{T-1} M_j \Delta_j^i
    \]
  for $t=1,\ldots,T-2$, $i=1,\ldots,P$,
    and
    \[
    w_0^{i+1} - w_0^{i} 
    \le - \frac{\tau \mu_{1}}{2}(\Delta_{1}^{i})^2 
    + 2(1-\tau) \sum_{j=1}^{T-1} M_j \Delta_j^{i}
    \]
    for $i=1,\ldots,P$.
    These are exactly the recurrences used in Lemma~\ref{prooflemreccruc}, which provides, using this lemma:
    \begin{equation}\label{finallemanter}
    r^{P+1} - 20 \varepsilon \sum_{t=1}^{T-1} \frac{M_t^2}{\mu_t} \le \tau^P \cdot r^1,
    \end{equation}
    where 
    $$r^i := \sum_{t=0}^{T-2} \varepsilon^{2^{T-1} - 2^{T-1 - t}} w_t^i.$$
    Since all terms in the sum defining $r^{P+1}$ are nonnegative, we can bound
    \begin{equation}\label{bounde0kp}
    r^{P+1} \ge w_0^{P+1}=e_0^{k_{P+1}}.
    \end{equation}
    Moreover, we have shown in the proof of Lemma~\ref{prooflemreccruc} that
$$
\sum_{t=0}^{T-2} \varepsilon^{2^{T-1} - 2^{T-1 - t}} \leq 2,
    $$
    which gives
\begin{equation}\label{finbcruc}
\qquad r^1 \leq  2 \sup_{t}  w_{t}^{1} \le 2c
  \end{equation}  
    where the last inequality comes from \eqref{lemPHPineq1}
    for constant $c$
    defined in \eqref{defc}.
    Therefore,
    \[
 e_0^{k_{P+1}} = w_0^{P+1} \stackrel{\eqref{bounde0kp}}{\le} r^{P+1} 
 \stackrel{\eqref{finbcruc}, \eqref{finallemanter}}{\le} 20 \varepsilon \sum_{t=1}^{T-1} \frac{M_t^2}{\mu_t} + 2c \cdot \tau^P = \frac{\bar \varepsilon}{2} + 2c \cdot \tau^P.
    \]
\end{proof}

\vspace{0.2cm}

To obtain probabilistic guarantees in our complexity analysis, we require a lower bound on the conditional probability that event \( E^k \) holds, given the history of previous iterations. The next lemma ensures that this probability remains uniformly bounded below, which is crucial for controlling the expected number of successful iterations.

\begin{lemma}
    Consider the events \( E^k \) defined in \eqref{defEventEk}. Then, for all \( k \ge 2 \),
    \[
    \mathbb{P}\left(E^k \,\middle|\, \sigma \left(E^1, \ldots, E^{k-1} \right) \right) \ge p := \prod_{t=1}^{T-1} p_{t+1}
    \]
    where
    \begin{equation}
p_{t} = \min_{\tilde \xi_t \in \Xi_t} \mathbb{P}(\boldsymbol \xi_t = \tilde \xi_t).
    \end{equation}
\end{lemma}

\begin{proof}
    Let us define the individual stagewise events
    \[
    E_t^k := \left\{ \tilde e_t^k \ge e_t^k \right\}.
    \]
    Since \( \boldsymbol \xi_{t+1}^k \) is independent of the sample history \( \boldsymbol \xi_{[t]}^{[k]} \), we can write:
    \[
    \mathbb{P}\left( E_t^k \,\middle|\, \sigma \left( \boldsymbol \xi_{[t]}^{[k]} \right) \right) = 
    \mathbb{P} \left( \left[ Q_t(x_t^k; \cdot) - V_t^k(x_t^k; \cdot) \right] (\boldsymbol \xi_{t+1}^k) \ge \left[ \mathcal{Q}_t - \mathcal{V}_t^k \right]\left(x_t^k \right) \,\middle|\, \sigma \left( \boldsymbol \xi_{[t]}^{[k]} \right) \right) \ge p_{t+1}.
    \]
    This implies that, conditioned on the past sample paths, each individual event \( E_t^k \) occurs with probability at least \( p_{t+1} \). Consequently, we can bound the probability of their joint occurrence:
    \begin{align*}
        \mathbb{P}\left( \bigcap_{t=1}^{T-1} E_t^k \,\middle|\, \sigma \left( \boldsymbol \xi_{[T]}^{[k-1]} \right) \right)
        &= \prod_{t=1}^{T-1} \mathbb{P}\left( E_t^k \,\middle|\, \sigma \left( \boldsymbol \xi_{[T]}^{[k-1]}, \bigcap_{j=1}^{t-1} E_j^k \right) \right) \\
        &= \prod_{t=1}^{T-1} \mathbb{E}\left[ \mathbb{P}\left( E_t^k \,\middle|\, \sigma \left( \boldsymbol \xi_{[t]}^{[k]} \right) \right) \,\middle|\, \sigma \left( \boldsymbol \xi_{[T]}^{[k-1]}, \bigcap_{j=1}^{t-1} E_j^k \right) \right] \\
        &\ge \prod_{t=1}^{T-1} p_{t+1}.
    \end{align*}
    Since \( \sigma(E^1, \ldots, E^{k-1}) \subset \sigma( \boldsymbol \xi_{[T]}^{[k-1]} ) \), the desired result follows from the law of total probability.
\end{proof}

\vspace{0.2cm}

To quantify the number of BSDDP iterations required to achieve a given accuracy level, we must analyze the number of times certain favorable events \( E^k \) occur. The lemma below provides a probabilistic bound for the stopping time \( K \), defined as the first iteration at which at least \( R \) such events have happened. The result shows that if each event has a uniform lower bound on its conditional probability of occurring, then the expected stopping time is finite and controlled. Additionally, a high-probability tail bound is provided to guarantee concentration around the expected value.

\begin{lemma}\label{lemRuns}
Let $(E^k)_{k \ge 1}$ be a sequence of events in a probability space $(\Omega, \mathcal{F}, \mathbb{P})$ such that for all $k \ge 1$,
\[
\mathbb{P}(E^{k+1} \mid \mathcal{F}_k) \ge p \quad \text{a.s.,}
\]
where $\mathcal{F}_k := \sigma(E^1, \dots, E^k)$ and $p \in (0,1]$ is fixed. For a nonnegative integer $R$, define 
$$K := \inf\{n \ge R : \text{all events } E^{n-R+1}, \dots, E^n \text{ occur}\},$$
to be the first index for which $R$ consecutive events have occurred. Then:
\begin{enumerate}
    \item The expectation of $K$ is bounded by
    $$
    \mathbb{E}[K] \le \sum_{i=1}^R \frac{1}{p^i}.
    $$

    \item For any $n \ge R$, the tail probability is bounded by
    $$
    \mathbb{P}(K > n) \le (1 - p^R)^{\lfloor n/R \rfloor}.
    $$
\end{enumerate}
\end{lemma}

\if{
-----------

\[
X_k = \mbox{\# of successfull consecutive strictly preceding $k$}
\]
Have
\[
X_0=0
\]
\[
\mu_R = \E(k | X_k = R-1, X_0=0) = \sum_{k=R-1}^\infty k P_{0,R-1}^k
\]
\[
P(X_{k+1}=j+1| X_k =j ) =p \quad
P(X_{k+1}=0| X_k =j ) =1-p
\]
\[
P^{k+1}_{0,j+1} = P( X_{k+1}=j+1 | X_0 =0)
\]
\[
P^{k+1}_{0,j+1} =
\sum_i P(X_{k+1}=j+1 | X_k =i) P(X_k=i | X_0=0) =
p P(X_k=j | X_0=0) = p P^k_{0,j}
\]
\begin{align*}
    P^{k+1}_{0,0} &= P (X_{k+1}=0 | X_0=0) =
\sum_i P(X_{k+1}=0 | X_k =i)P(X_k=i | X_0=0) \\
&=
(1-p) \sum_{i>0} P(X_k=i | X_0=0) = (1-p)
\end{align*}

\begin{align*}
   (k+1) P^{k+1}_{0,j+1}
   = (k+1) p P^k_{0,j}
   = p ( k P^k_{0,j}) + p P^k_{0,j}
\end{align*}
For any $j \ge 1$, have:
\[
\mu_{j+1} = \sum_{k=j-1}^\infty 
(k+1) P^{k+1}_{0,j}
= p \sum_{k=j-1}^\infty ( k P^k_{0,j-1}) + p
= p \mu_j + p
\]
\[
\mu_1 =  \sum_{k=-1}^\infty 
(k+1) P^{k+1}_{0,0} =(1-p) 
\]
\[
\mu_1=p \mu_0 + p
\]
\[
\mu_2 = p \mu_1 + p =
p (p \mu_0+p) + p =
p^2 \mu_0 + p^2 + p
\]
\[
\mu_R = p^R \mu_0 + p^R + \ldots +p = p^R \mu_0 +
p (1+\ldots+ p^{R-1}) =

\]

--------------
}\fi

\begin{proof}
\begin{enumerate}
\item Let $L_k$ denote the number of consecutive occurrences of the events $E^i$ ending at time $k$. Then $(L_k)_{k \ge 0}$ defines a Markov chain on the state space $\{0, 1, \dots, R\}$ with absorbing state $R$. Define $\mu_j$ as the expected time to reach state $R$ starting from $L_0 = j$. Our goal is to bound $\mu_0 = \mathbb{E}[K]$, with terminal condition $\mu_R = 0$.

    For $j \in \{0, \dots, R-1\}$, the expected time follows the recurrence:
    $$\mu_j = 1 + \mathbb{P}(E^{k+1}|\mathcal{F}_k) \mu_{j+1} + (1-\mathbb{P}(E^{k+1}|\mathcal{F}_k))\mu_0.$$
    Since $\mathbb{P}(E^{k+1}|\mathcal{F}_k) \ge p$ and $\mu_{j+1} \le \mu_j \le \mu_0$, we obtain
    $$\mu_j \le 1 + p \mu_{j+1} + (1-p)\mu_0 \iff (\mu_j-\mu_0) \le 1 + p(\mu_{j+1} - \mu_0).$$
    Therefore, for $j=0,\ldots,R-1$, we have
$$
p^j(\mu_j-\mu_0) \leq p^j + p^{j+1}(\mu_{j+1}-\mu_0).
$$
Summing the above inequalities for $j=0,\ldots,R-1$, we obtain
    $$p^0 (\mu_0 - \mu_0) - p^R(\mu_R-\mu_0) \le 1+p+\ldots+p^{R-1}.$$
    Recalling that $\mu_R=0$, it follows that the
    expectation $\mathbb{E}[K]$ is bounded by
    $$\mathbb{E}[K] = \mu_0 \le \frac{1+p+\ldots+p^{R-1}}{p^R}=\sum_{j=1}^R \frac{1}{p^j}.$$
\item Define disjoint blocks of size $R$ by
$$
B_i := \bigcap_{k=(i-1)R+1}^{iR} E^k, \quad i \in \mathbb{N},
$$
and let $N := \inf\{i \ge 1 : B_i \text{ occurs}\}$ be the index of the first successful block. The event $\{N > m\}$ means that all of the first $m$ blocks failed, i.e., $\bigcap_{i=1}^m B_i^c$. Observe that
$$
\begin{aligned}
\mathbb{P}(N > m+1) &= \mathbb{P}\left(\bigcap_{i=1}^{m+1} B_i^c\right) \\
&= \mathbb{E}\left[\mathbf{1}_{\cap_{i=1}^m B_i^c} \cdot \mathbb{P}\left(B_{m+1}^c \mid \sigma(E^1,\dots,E^{mR})\right)\right].
\end{aligned}
$$
Since
$$
\mathbb{P}\left(B_{m+1} \mid \sigma(E^1,\dots,E^{mR})\right) \ge p^R \quad \text{a.s.},
$$
it follows that $\mathbb{P}\left(B_{m+1}^c \mid \sigma(E^1,\dots,E^{mR})\right) \le 1 - p^R$, and hence
$$
\mathbb{P}(N > m+1) \le (1-p^R)\,\mathbb{P}(N > m).
$$
Iterating this relation gives
$$
\mathbb{P}(N > m) \le (1-p^R)^m.
$$
Now set $m = \lfloor n/R \rfloor$. The event $\{K > n\}$ implies that none of the first $m$ blocks succeeded, i.e., $\{K > n\} \subseteq \bigcap_{i=1}^m B_i^c$. Therefore,
$$
\mathbb{P}(K > n) \le \mathbb{P}\left(\bigcap_{i=1}^m B_i^c\right) \le (1-p^R)^m = (1-p^R)^{\lfloor n/R \rfloor}.
$$
\end{enumerate}
\end{proof}

The result below establishes a complexity bound for the BSDDP algorithm under strong convexity assumptions. It provides a non-asymptotic guarantee on the number of iterations required to reach a prescribed accuracy level at the first stage. To the best of our knowledge, this is the first result offering an explicit iteration bound for BSDDP in the strongly convex setting, with a closed-form dependence on the model's curvature parameters, the sample space, and the desired accuracy. This theorem also offers insights into how problem structure and sampling probabilities influence convergence speed, which may inform both algorithm design and parameter tuning in practice.

\begin{theorem}[Complexity of BSDDP for strongly convex problems] \label{complexitysddp}
Let $K$ denote the first iteration index of BSDDP satisfying
 \(e_0^K \le \bar \varepsilon.\)
    % Let $K$ denote the number of iterations performed by BSDDP, with parameter $\tau_0=\tau$ defined in $\eqref{defTau}$, before it finds an index $k$ with \(e_0^k \le \bar \varepsilon.\)
    Then we have 
    \begin{equation}
    \mathbb{E}[K] \le \sum_{i=1}^R \frac{1}{p^i}, \quad \text{for} \quad R = \displaystyle P \cdot \prod_{t=2}^T N_t + 1,\end{equation} 
    where 
    \begin{equation*}
    P= \left\lceil\left(
        1+ \left[ \frac{40}{\bar \varepsilon} \sum_{t=1}^{T-1} \frac{M_t^2}{\mu_t} \right]^{2^{T-1}-1}
    \right) \log\left( \frac{4}{\bar \varepsilon} \sum_{t=1}^{T-1} M_t D_t \right) \right \rceil, \quad p = \prod_{t=2}^{T} \min_{\tilde \xi_{t} \in \Xi_{t}} \mathbb{P}(\boldsymbol \xi_{t}=\tilde \xi_t).
    \end{equation*}
    In addition, for any $n \ge 1$, we have
    \begin{equation}\label{tailbound}
    \mathbb{P}(K > n) \le (1 - p^R)^{\lfloor n/R \rfloor}.
    \end{equation}
\end{theorem}

\if{
-------------

\[
\E[K] \le \frac1p 
\left(\frac{p^{-R}-1}{p^{-1}-1} \right)
= \left(\frac{p^{-R}-1}{1-p} \right) \le
\frac{1}{(1-p)p^R}
\]
--------------

}\fi

\begin{proof}
    Let $L$ denote the number of iterations required for the event $E^k$ (as defined in \eqref{defEventEk}) to occur 
    \[
    R := P \cdot \prod_{t=2}^T N_t + 1
    \]
    consecutive times. By Lemma~\ref{lemRuns}, we have 
    \(\mathbb{E}[L] \le \sum_{i=1}^R 1/p^{i}.\)

    Let $S \subset \mathbb{N}$ denote the set of $R$ consecutive iterations in which the event $E^k$ occurs, i.e., $S = \{k - R + 1, \ldots, k\}$ for some stopping index $k = L$. Since the joint realization $(\xi_2^k, \ldots, \xi_T^k)$ takes at most $\prod_{t=2}^T N_t$ distinct values, and $|S| = R = P \cdot \prod_{t=2}^T N_t + 1$, the pigeonhole principle implies that some realization $\zeta$ must occur at least $P+1$ times within $S$.
    
    Let $A := \{k_1, \ldots, k_{P+1}\} \subset S$ be the ordered set of indices for which $(\xi_2^{k_i}, \ldots, \xi_T^{k_i}) = \zeta$
    for all
    $i=1,\ldots,P+1$. Since the indices are consecutive and the realization is repeated, the algorithm ensures that for each $i \ge 2$, we have $k_{i-1} = \ell(k_i) < k_i$, by the dictionary lookup mechanism in Step 2 of BSDDP.
    
    Applying Lemma~\ref{lemPHP} to this subset \(A\), we obtain
    \[
    \min_{j \le L} e_0^j \le e_0^{k_{P+1}} \le \frac{\bar \varepsilon}{2} + 2c \cdot \tau_0^P.
    \]
    Recalling the expression for $\tau_0$, we note that
    \[
    \frac{1}{1 - \tau_0} = 1 + \varepsilon^{-(2^{T-1}-1)} = 1 + \left[ \frac{40}{\bar \varepsilon} \sum_{t=1}^{T-1} \frac{M_t^2}{\mu_t} \right]^{2^{T-1}-1},
    \]
    thus 
    \[
    P \ge \frac{1}{1 - \tau_0} \log\left( \frac{4c}{\bar \varepsilon} \right) \implies (\tau_0 - 1)P \le \log\bigg( \frac{\bar \varepsilon}{4c} \bigg) \implies 2c \cdot e^{(\tau_0 - 1)P} \le \frac{\bar \varepsilon}{2}.
    \]
    By the standard exponential bound $\tau_0 \le e^{\tau_0 - 1}$, this ensures that $\displaystyle \min_{j \le L} e_0^j \le \bar \varepsilon$, which shows that an index \(k \le L\) satisfies the desired accuracy, i.e., \(K \le L\)
    and
    \(\mathbb{E}[K] \leq \mathbb{E}[L] \le \sum_{i=1}^R 1/p^{i}.\)

    The tail bound
\eqref{tailbound}
    follows directly from the second part of Lemma~\ref{lemRuns}.
\end{proof}

\vspace{0.2cm}

\begin{rem}
    Theorem \ref{complexitysddp} provides an explicit non-asymptotic upper bound on the number of BSDDP iterations required to reach a prescribed target accuracy $\bar \varepsilon$. The bound depends only on problem-specific quantities (Lipschitz and strong convexity constants) and the desired accuracy, but not on the state-space dimension. This shows the algorithm's scalability and theoretical efficiency for structured multistage problems.
\end{rem}

\begin{rem}
    If the sampling distributions are uniform over $\Xi_t$ for all stage $t$, as in \cite{lan2020}, then the lower bound $p$ becomes $p = \left( \prod_{t=1}^{T-1} N_{t+1}^{-1} \right)$. In this case, the complexity bound simplifies to
    \[
    \mathbb{E}[K] \le 
\sum_{i=1}^R    \left(\prod_{t=2}^T N_t \right)^i 
    \le
    \left( 1+  P  \prod_{t=2}^T N_t \right) \left( \prod_{t=2}^T N_t \right)^R,
    \]
    which asymptotically matches the complexity $1+P$ of DDP given in  
    \cite{complddpscpb}
    when each stage has a single scenario, i.e., $N_t = 1$ for all $t$.
\end{rem}

\section{Conclusion}

In this paper, we proposed a variant of SDDP called
BSDDP (Bidirectional SDDP) to solve strongly convex
multistage stochastic programs. We proved a complexity
bound for this solution method independent on the dimension
of the state vectors which improves the previous complexity
bound obtained in \cite{lan2020} for SDDP.
The bound also decreases as the constants of strong convexity
increase.

It would be interesting to extend this analysis to the study
of the complexity of other variants of SDDP such as 
\cite{baucke18} or incorporating inexactness
in the solutions of the subproblems
along iterations as in \cite{guigues2016isddp}.

\bibliographystyle{plain}
\bibliography{ref}

\section{Appendix}

We collect in this appendix two simple technical lemmas of convex analysis.

\begin{lemma}\label{subgradValueFunc1}
Let
$G: \mathbb{R}^m \times \mathbb{R}^n \rightarrow \mathbb{R}$ be a convex
function and let $Y$ be a convex
compact set. Then, the value function 
 defined as
\begin{equation}\label{pbdefsc0}
            \V(x) := \min_{y\in Y} \, G(y,x) \quad \forall \ x \in \R^n
\end{equation}
is convex, finite everywhere, and $\partial \V(x) \ne \emptyset$ for every $x \in \R^n$.
Moreover,  the following three conditions are equivalent:
\begin{itemize}
    \item[a)] 
    $s \in \partial \V(x)$;
    \item[b)] there exists an optimal solution $y_x$ of \eqref{pbdefsc0} such that $(0,s) \in \partial G(y_x,x)$, in which case
\[
s \in \partial_x G(y_x,x);
\]
\item[c)]
$s$
is a Lagrange multiplier for
the constraint $z=x$ in the following reformulation of
the value function $\V$:
\begin{align}\label{pbdefsc1}
            \V(x) = \min \{ G(y,z): y \in Y,\; z=x\}.
            \end{align}
\end{itemize}
\end{lemma}
\begin{proof}
It follows from Proposition 3.3.3 of \cite{bertsekas2009convex} that $\V$ is a convex function. Since for every $x \in \R^{n}$, $G(\cdot,x)$ is continuous on the
compact set $Y$, function $\V$ is finite-valued
everywhere and therefore $\partial \V(x) \neq \emptyset$.

The equivalence between a) and c) follows from \cite[Theorem 9.3]{rockafellar1993lagrange}, as \eqref{pbdefsc0} has a feasible solution by compactness of $Y$ and $\V$ is lower semicontinuous since it is finite and convex.
Using Theorem 24(a) of \cite{rockafellar1974conjugate}, we also have
that a) is equivalent to the existence of
an optimal solution $y_x$  of \eqref{pbdefsc0} such that $(0,s)\in \partial G(y_x,x)$; equivalently, for all $(y,z)$,
$$
G(y,z)\ge G(y_x,x)+\left \langle 0,\,y-y_x\right \rangle+\left \langle s,\,z-x\right \rangle
= G(y_x,x)+\left \langle s,\,z-x\right \rangle.
$$
In particular, setting $y=y_x$ gives 
$$
G(y_x,z)\ge G(y_x,x)+\left \langle s,\,z-x\right \rangle
$$
for all $z$,
which is precisely the subgradient inequality characterizing $s\in \partial_x G(y_x,x)$.
\end{proof}

\begin{lemma}\label{subgradValueFunc2}
Let $X \subset \mathbb{R}^n$ 
and $Y \subset \mathbb{R}^m$
be convex compact sets and let
$F: \mathbb{R}^m \times \mathbb{R}^n \rightarrow \mathbb{R}$ and 
$\Gamma: \mathbb{R}^m \rightarrow \mathbb{R}$ be convex
functions.
 Then, the value function 
 defined as
\begin{equation}\label{pbdefsc2}
            \V(x) := \min_{y\in Y} \, \left\{G(y,x):=F(y,x) + \Gamma(y) \right\} \quad \forall \ x \in \R^n,
\end{equation}
is convex, finite everywhere, and $\partial \V(x) \ne \emptyset$ for every $x \in \R^n$.

If, in addition,
there exists $0 \leq M < +\infty$
such that
\begin{equation}\label{syx}
    \|s_x(y,x)\| \le M,
    \end{equation}
for every $(y,x) \in Y \times X$ 
and every subgradient 
$s_x(y,x) \in \partial_x F(y,x)$, 
    then $\V(\cdot)$
    restricted to $X$
        is $M$-Lipschitz continuous
        or equivalently $\|s\| \le M$ for every  $s \in \partial \V(x)$ and $x \in X$.
\end{lemma}

\begin{proof} We have that $\V$ is convex, finite everywhere, and $\partial \V(x) \ne \emptyset$ for every $x \in \R^n$ by a direct application of Lemma \ref{subgradValueFunc1} to $G(y,x):=F(y,x) + \Gamma(y).$ 

For the second part, assume that $\|s_x(y,x)\| \le M$, for every $(y,x) \in Y \times X$ and every subgradient $s_x(y,x) \in \partial_x F(y,x) = \partial_x G(y,x)$. Thus, for any $x \in X$ and $s \in \partial \V(x),$ Lemma \ref{subgradValueFunc1} shows that $s \in \partial_x G(y_x,x)$ for some $y_x \in Y$, so that $\|s\| \le M.$ 

To show the last equivalence, we proceed as follows. Assume $\|s\|\le M$ for all $x\in X$ and $s\in\partial \V(x)$. Then for any $x,x'\in X$ and $s\in\partial \V(x)$,
\[\V\left(x'\right)-\V(x)\ge \left\langle s,x'-x\right \rangle \ge -\|s\|\,\left\|x'-x\right\|\ge -M\left\|x'-x\right\|.\] 
Switching $x,x'$ yields
$|\V(x')-\V(x)|\le M\|x'-x\|$, i.e., $\V$ is $M$-Lipschitz on $X$ with respect to \ $\|\cdot\|$.
Conversely, if $\V$ is $M$-Lipschitz on $X$ w.r.t.\ $\|\cdot\|$, then for any $x\in X$, $s\in\partial \V(x)$, and $x'\in X$,
\[\left\langle s,x'-x\right \rangle \le \V\left(x'\right)-\V(x)\le M\left\|x'-x\right\|,\] 
hence $\|s\|\le M$.
\end{proof}

\end{document}